\documentclass[12pt]{amsart}
\usepackage{amssymb,amsmath,amsthm,latexsym,verbatim}
\usepackage{mathrsfs,url}
\usepackage{a4wide}
\usepackage[usenames,dvipsnames]{color}
\usepackage{dsfont}
\usepackage[utf8]{inputenc}
\usepackage[T1]{fontenc}
\usepackage{enumerate,enumitem}
\usepackage[mathscr]{euscript}
\usepackage{mathabx}
\usepackage{xcolor}

\usepackage{anysize}
\marginsize{1in}{1in}{0.8in}{1in}

\newcounter{smallromans}

\newenvironment{romanenumerate}
{\begin{list}{{\normalfont\textrm{(\roman{smallromans})}}}%
  {\usecounter{smallromans}\setlength{\itemindent}{0cm}%
   \setlength{\leftmargin}{5.5ex}\setlength{\labelwidth}{5.5ex}%
   \setlength{\topsep}{.5ex}\setlength{\partopsep}{.5ex}%
   \setlength{\itemsep}{0.1ex}}}%
{\end{list}}

\newcounter{smallromansdash}

{\end{list}}

\newcounter{bigromans} 
  {\end{list}}

\newtheorem{theorem}{Theorem}[section]

\newtheorem{lemma}[theorem]{Lemma}

\newcounter{maintheorem}

\newtheorem{mainth}[maintheorem]{Theorem}

\theoremstyle{definition}
\newtheorem{definition}[theorem]{Definition}
\newtheorem{example}[theorem]{Example}

\theoremstyle{remark}

\numberwithin{equation}{section}
\def\sqr#1#2{{\,\vcenter{\vbox{\hrule height.#2pt\hbox{\vrule width.#2pt
height#1pt \kern#1pt\vrule width.#2pt}\hrule height.#2pt}}\,}}

\begin{document}
\title{Quantifying shrinking and boundedly complete bases}

\author{Dongyang Chen}
\address{School of Mathematical Sciences\\ Xiamen University,
Xiamen, 361005, China}
\email{cdy@xmu.edu.cn}

\author{Tomasz Kania}
\address{Mathematical Institute\\Czech Academy of Sciences\\\v Zitn\'a 25 \\115 67 Praha 1, Czech Republic\\
and\\  Institute of Mathematics and Computer Science\\ Jagiellonian University\\
{\L}ojasiewicza 6\\ 30-348 Krak\'{o}w, Poland}
\email{kania@math.cas.cz, tomasz.marcin.kania@gmail.com}

\author{Yingbin Ruan}
\address{College of Mathematics and Informatics\\ Fujian Normal University,
Fuzhou, 350007, China}
\email{yingbinruan@sohu.com}


\thanks{Dongyang Chen was supported by the National Natural Science Foundation of China (Grant No. 11971403) and the Natural Science Foundation of Fujian Province of China (Grant No. 2019J01024). Tomasz Kania acknowledges with thanks funding received from SONATA 15 No. 2019/35/D/ST1/01734.}

\begin{abstract}
We investigate possible quantifications of R.~C.~James' classical work on bases and reflexivity of Banach spaces. By introducing new quantities measuring how far a basic sequence is from being shrinking and/or boundedly complete, we prove quantitative versions of James' famous characterisations of reflexivity in terms of bases. Furthermore, we establish quantitative versions of James' characterisations of reflexivity of Banach spaces with  unconditional bases.
\end{abstract}

\date{\today}
\maketitle

\baselineskip=18pt 	
\section{Introduction}

James' classical paper \cite{James:1950} linking reflexivity and bases is deeply entrenched in modern Banach space theory; the now standard characterisation of reflexivity in terms of shrinkingness and bounded completeness of bases/basic sequences is proved therein. For a space with a basis a trade-off between various measures of (non-)weak compactness of the unit ball and closedness of the basis to be simultaneously shrinking and boundedly complete is naturally expected. In the present paper we investigate possible quantifications of the said notions with the aim of establishing quantitative analogues of James' criteria for reflexivity expressed in terms of bases/basic sequences. This line of research is particularly timely in the light of numerous recent results in this spirits (see, \emph{e.g.}, \cite{BKS, KKS, KPS,KS}).

In order to quantify James' criteria of reflexivity, it is thus necessary to introduce quantities measuring how far a~basis is from being shrinking and/or boundedly complete. In Section~\ref{sec:shrinking}, a~quantity $\textrm{sh}((x_n)_{n=1}^\infty)$ measuring how far a basic sequence $(x_n)_{n=1}^\infty$ is from being shrinking is introduced and investigated. In Section~\ref{sect:boundedly}, we introduce three equivalent quantities $\textrm{bc}_{1},\textrm{bc}_{2}$, and $\textrm{bc}_{3}$ measuring (non-)bounded completeness of a~basis. Besides the quantities $\textrm{sh}$ and $\textrm{bc}$, we also need a quantity $\textrm{sep}$ measuring non-separability of a set, a quantity $\alpha_{Y}(X)$ measuring how well a Banach space $Y$ is from being isomorphically embedded into another Banach space $X$, and two important mutually equivalent quantities $\textrm{wk},\textrm{wck}$ measuring weak non-compactness of sets. In this paper, we quantify Theorems \ref{1.1}--\ref{1.4}, respectively. In order to state the results let us introduce the following conventions. 

If $X$ is a~Banach space with a basis $(x_{n})_{n=1}^\infty$, we denote by $(x^{*}_{n})_{n=1}^{\infty}$ the sequence of coordinate functionals associated with the basis and by $K$ the basis constant. When $(x_{n})_{n=1}^\infty$ is unconditional, we denote by $K_u$ the unconditional constant of $(x_{n})_{n=1}^\infty$. Slightly abusing the notation, for a Banach space $X$ with a fixed basis $(x_n)_{n=1}^\infty$, we set 
\[
    V=[x^{*}_{n}\colon n\in \mathbb{N}].
\]

\begin{mainth}\label{Thm:A}
Let $X$ be a Banach space with a basis $(x_{n})_{n=1}^\infty$. If $K$ is the basis constant, then
\[
    \operatorname{sh}((x_{n})_{n=1}^\infty)\leq \widehat{\operatorname{d}}(B_{X^{*}},V)\leq (K+1)\operatorname{sh}((x_{n})_{n=1}^\infty),
\]
and
\[
    \frac{1}{2K}\operatorname{bc}_{2}((x^{*}_{n})_{n=1}^{\infty})\leq \operatorname{sh}((x_{n})_{n=1}^\infty)\leq K\operatorname{bc}_{2}((x^{*}_{n})_{n=1}^{\infty}).
\]
\noindent If, in addition, $(x_{n})_{n=1}^\infty$ is unconditional, then
\[
    \frac{1}{K_{u}}\operatorname{sh}((x_{n})_{n=1}^\infty)\leq\alpha_{\ell_{1}}(X)\leq \operatorname{sep}(B_{X^*})\leq \widehat{\operatorname{d}}(B_{X^{*}},V).
\]
\end{mainth}

\begin{mainth}\label{Th:B}
Let $X$ be a Banach space with a basis $(x_{n})_{n=1}^\infty$.  If $K$ is the basis constant, then
\[
    \operatorname{sh}((x^{*}_{n})_{n=1}^{\infty})\leq \operatorname{bc}_{2}((x_{n})_{n=1}^\infty)\leq 2K^{2}\operatorname{sh}((x^{*}_{n})_{n=1}^{\infty}).
\]
\noindent If, in addition, $(x_{n})_{n=1}^\infty$ is unconditional, then
\[
    \frac{1}{K_{u}}\alpha_{c_{0}}(X)\leq \operatorname{bc}_{1}((x_{n})_{n=1}^\infty)\leq K_{u}^{3}\alpha_{c_{0}}(X).
\]
\end{mainth}

\begin{mainth}\label{Th:C}
Let $X$ be a Banach space with a basis $(x_{n})_{n=1}^\infty$.
\begin{romanenumerate}
    \item\label{item1:thC} If $(x_{n})_{n=1}^\infty$ is boundedly complete, then
        \[ \left\{
            \begin{array}{lcl}
            \operatorname{sh}((x_{n})_{n=1}^\infty)&\leq & 4K^{3}\operatorname{wk}_{X}(B_{X})\\
            \operatorname{wck}_{X}(B_{X})& \leq & (K+1)\widehat{\operatorname{d}}(B_{X^{*}},V)
            \end{array} \right.
        \]

    \item \label{item2:thC} If $(x_{n})_{n=1}^\infty$ is shrinking, then
        \[ \left\{
            \begin{array}{lcl}
            \operatorname{bc}_{3}((x_{n})_{n=1}^\infty)& \leq & 2K^{2}\operatorname{wk}_{X}(B_{X})\\
            \operatorname{wck}_{X}(B_{X})& \leq & (K+1)^{2}\operatorname{bc}_{2}((x_{n})_{n=1}^\infty).
            \end{array} \right.
        \]
\end{romanenumerate}
\end{mainth}

\begin{mainth}\label{Th:D}
Let $X$ be a Banach space with an unconditional basis $(x_{n})_{n=1}^\infty$.
\begin{romanenumerate}
    \item \label{item1:ThD}If $X$ contains no isomorphic copies of $\ell_{1}$, then
        \[
            \frac{1}{K_{u}^{3}K(K+1)^{2}}\operatorname{wck}_{X}(B_{X})\leq \alpha_{c_{0}}(X)\leq \alpha_{\ell_{1}}(X^{*})\leq\operatorname{wck}_{X}(B_{X}).
        \]
    \item \label{item2:ThD} If $X$ contains no isomorphic copies of $c_{0}$, then
        \[
            \frac{1}{K_{u}(K+1)^{2}}\operatorname{wck}_{X}(B_{X})\leq \alpha_{\ell_{1}}(X)\leq \operatorname{wck}_{X}(B_{X}).
        \]
    \item \label{item3:ThD}
        \[
            \frac{1}{K_{u}^{3}K(K+1)^{2}}\operatorname{wck}_{X}(B_{X})\leq\operatorname{sep}(B_{X^{**}})\leq \operatorname{wk}_{X}(B_{X}).
        \]
\end{romanenumerate}
\end{mainth}

\section{Preliminaries}

Our notation and terminology are standard and mainly follow \cite{Sin} and \cite{LT}. Throughout this paper, all Banach spaces are infinite-dimensional and real for the sake of convenience. By a \emph{subspace} we understand as a closed, linear subspace and by an \emph{operator} we mean a bounded, linear operator. An operator $T\colon X\rightarrow Y$ is \textit{bounded below}, whenever there is $\gamma > 0$ such that $\|Tx\|\geq \gamma\|x\|$ ($x\in X$); equivalently, when $T$ is an isomorphism onto its range. If $X$ is a Banach space, we denote by $B_{X}$ its closed unit ball and by $I_{X}$ the identity operator on $X$.
For a subset $A$ of $X$, $[A]$ stands for the closed linear span of $A$.

\subsection{Basics on Schauder bases}

A sequence $(x_{n})_{n=1}^\infty$ in a Banach space $X$ is called a~\textit{(Schauder) basis} for $X$ whenever every $x\in X$ has a unique expansion $x=\sum_{n=1}^{\infty}a_{n}(x)x_{n}$ for some scalar sequence $(a_{n}(x))_{n=1}^\infty$. A sequence $(x_{n})_{n=1}^\infty$ in a Banach space $X$ is called \textit{basic} if it is a basis for $[x_{n}\colon n\in \mathbb{N}]$. For every $n\in N$, the linear functional $x^{*}_{n}$ on $X$ given by $\langle x^{*}_{n},x\rangle=a_{n}(x)$ ($x\in X$) is well-defined and bounded. We call $(x^{*}_{n})_{n=1}^{\infty}$ the \textit{biorthogonal functionals} associated to the basis $(x_{n})_{n=1}^\infty$.

The canonical basis projections $(P_{n})_{n=1}^\infty$ associated to the basis $(x_{n})_{n=1}^\infty$ are given by $P_{n}(x)=\sum_{i=1}^{n}\langle x^{*}_{i},x\rangle x_{i}$ ($x\in X$). Since each functional $x^{*}_{n}$ is continuous, so is $P_{n}$ ($n\in \mathbb{N}$). The Uniform Boundedness principle implies that $K:=\sup_{n}\|P_{n}\|<\infty$. The number $K$ is called the \textit{basis constant} of $(x_{n})_{n=1}^\infty$. We have $P_{n}^{*}x^{*}=\sum_{i=1}^{n}\langle x^{*},x_{i}\rangle x^{*}_{i}$ ($ x^{*}\in X^{*}$). In particular, $(x^{*}_{n})_{n=1}^{\infty}$ is a basic sequence with basis constant at most $K$.
We denote by $j$ be the canonical map from $X$ to $V^{*}$ defined by $\langle jx,x^{*}\rangle=\langle x^{*},x\rangle $ for all $x\in X$ and $x^{*}\in V$. Then $j$ is bounded below as
\begin{equation}\label{2}
    \frac{1}{K}\|x\|\leq \|jx\|\leq \|x\| \quad (x\in X).
\end{equation}
and  $(jx_{n})_{n}$ is the sequence of biorthogonal functionals associated to the basic sequence $(x^{*}_{n})_{n=1}^{\infty}$. We let $W=[jx_{n}\colon n\in \mathbb{N}]$.


Let $X$ be a Banach space with an unconditional basis $(x_{n})_{n=1}^\infty$. Then for every choice of unit scalars $\theta=(\theta_{n})_{n=1}^\infty$ the map $M_{\theta}\colon X\rightarrow X$ defined by $M_{\theta}(x)=\sum_{n=1}^{\infty}\theta_{n}\langle x^{*}_{n},x\rangle x_{n}$ ($x\in X$) is continuous. The Uniform Boundedness principle implies that $\sup_{\theta}\|M_{\theta}\|$ is finite. The number $K_u:=\sup_{\theta}\|M_{\theta}\|$ is called the \textit{unconditional constant} of $(x_{n})_{n=1}^\infty$. One observes readily that $K_{u}\geq K$.

Let $X$ be a Banach space with a basis $(x_{n})_{n=1}^\infty$. Then $(x_{n})_{n=1}^\infty$ is
\begin{itemize}
    \item \textit{shrinking} if the sequence of biorthogonal functionals $(x^{*}_{n})_{n=1}^{\infty}$ is a basis for $X^{*}$, \emph{i.e.}, when $X^{*}=V$.
    \item \textit{boundedly complete} if for every scalar sequence $(a_{n})_{n=1}^{\infty}$ with $\sup_{n}\|\sum_{i=1}^{n}a_{i}x_{i}\|<\infty$, the series $\sum_{n=1}^{\infty}a_{n}x_{n}$ converges.
\end{itemize}

We shall use the notation $\|x^{*}\|_{n}=\|x^{*}|_{[x_{i}\colon i>n]}\|, \quad (x^{*}\in X^{*}, n\in \mathbb{N}),$ which renders to
\begin{equation}\label{1}
\|x^{*}\|_{n}=\textrm{d}(x^{*},[x^{*}_{i}\colon i\leq n]).
\end{equation}
(see \cite[Proposition 4.1]{Sin}).

It is known that a basis $(x_{n})_{n=1}^\infty$ is shrinking if and only if $\|x^{*}\|_{n}\to 0$ as $n\to\infty$ for every $x^{*}\in X^{*}$. A basic sequence $(x_{n})_{n=1}^\infty$ in a Banach space $X$ is  \textit{shrinking} if it is a~shrinking basis for $[x_{n}\colon n\in\mathbb{N}]$. Similarly, a~basic sequence $(x_{n})_{n=1}^\infty$ is called \textit{boundedly complete} if it is a boundedly complete basis for $[x_{n}\colon n\in\mathbb{N}]$.

\subsection{The James space} For illustratory purposes, we shall occasionally invoke the order-one quasi-reflexive \emph{James space} $\mathcal{J}$, which is a sequence space comprising all scalar sequences $x=(a_{n})_{n=1}^\infty$ for which $\lim_{n\rightarrow \infty}a_{n}=0$ and
\[
    \|x\|_{\mathcal{J}}=\frac{1}{\sqrt{2}}\sup [|a_{p_{1}}-a_{p_{2}}|^{2}+|a_{p_{2}}-a_{p_{3}}|^{2}+\ldots+|a_{p_{n}}-a_{p_{n+1}}|^{2}+|a_{p_{n+1}}-a_{p_{1}}|^{2}]^{\frac{1}{2}}<\infty,
\]
where the supremum is taken over all $n$ and all choices of integers $p_{1}<p_{2}<\ldots<p_{n+1}$.

The standard unit vectors $(e_{n})_{n=1}^\infty$ form a monotone, shrinking normalised basis for $\mathcal{J}$, whereas $(\sum_{i=1}^{n}e_{i})_{n=1}^\infty$ is a boundedly complete basis for $\mathcal{J}$.

\subsection{James' theorems}
Theorem~\ref{Thm:A} is directly motivated by the following theorem of James \cite{James:1950} that we shall now invoke.

\begin{theorem}\label{1.1}
Let $X$ be a Banach space with a basis $(x_{n})_{n=1}^{\infty}$. Then the following assertions are equivalent:
\begin{romanenumerate}
\item $(x_{n})_{n=1}^{\infty}$ is shrinking;
\item $X^{*}=[x^{*}_{n}\colon n\in \mathbb{N}]$;
\item $(x^{*}_{n})_{n=1}^{\infty}$ is a boundedly complete basic sequence.
\end{romanenumerate}
If in addition $(x_{n})_{n=1}^{\infty}$ is unconditional, then $(\operatorname{i})$--$(\operatorname{iii})$ are equivalent to $(\operatorname{iv})$--$(\operatorname{v})$:
\begin{romanenumerate}
\setcounter{smallromans}{3}
\item $X^{*}$ is separable;
\item $X$ contains no isomorphic copies of $\ell_{1}$.
\end{romanenumerate}
\end{theorem}

Similarly, Theorem~\ref{Th:B} quantifies the following theorem of James \cite{James:1950}.
\begin{theorem}\label{1.2}
Let $X$ be a Banach space with a basis $(x_{n})_{n=1}^{\infty}$. Then the following assertions are equivalent:
\begin{romanenumerate}
\item $(x_{n})_{n=1}^{\infty}$ is boundedly complete;
\item $(x^{*}_{n})_{n=1}^{\infty}$ is shrinking.
\end{romanenumerate}
If in addition $(x_{n})_{n=1}^{\infty}$ is unconditional, then $(\operatorname{i})$-$(\operatorname{ii})$ are equivalent to the following:
\begin{romanenumerate}
\setcounter{smallromans}{2}
\item $X$ contains no subspaces isomorphic to $c_{0}$.
\end{romanenumerate}
\end{theorem}
James \cite{James:1950} proved that joint occurrence of shrinkingness and bounded completeness characterise reflexivity, which is out main motivation for Theorem~\ref{Th:C}.
\begin{theorem}\label{1.3}
Let $X$ be a Banach space with a basis $(x_{n})_{n=1}^{\infty}$. Then $X$ is reflexive if and only if $(x_{n})_{n=1}^{\infty}$ is both boundedly complete and shrinking.
\end{theorem}

The following theorem is again due to R. C. James \cite{James:1950} except that the last statement that had been proved earlier by S. Karlin \cite{Kar} who employed different techniques.

\begin{theorem}\label{1.4}
Let $X$ be a Banach space with an unconditional basis. The following assertions are equivalent:
\begin{romanenumerate}
\item $X$ is reflexive.
\item No subspace of $X$ is isomorphic to either of $\ell_{1}$ or $c_{0}$.
\item No subspace of either $X$ or $X^{*}$ is isomorphic to $\ell_{1}$.
\item $X^{**}$ is separable.
\end{romanenumerate}
\end{theorem}

\subsection{Measures of weak non-compactness and non-separability}
Let $X$ be a Banach space and $A, B\subseteq X$ two non-empty sets. We set
\begin{itemize}
    \item $\textrm{d}(A,B)=\inf\{\|a-b\|\colon a\in A,b\in B\},$
    \item $\widehat{\textrm{d}}(A,B)=\sup\{\textrm{d}(a,B)\colon a\in A\}.$
\end{itemize}
$\textrm{d}(A,B)$ is the ordinary distance between $A$ and $B$ and $\widehat{\textrm{d}}(A,B)$ is the (non-symmetrised) Hausdorff distance from $A$ to $B$.
The quantity $\widehat{\textrm{d}}$ measures how much the set $A$ sticks out from the set $B$ and is used to define several measures of weak non-compactness. The present paper focuses on two standard equivalent measures. The first one is inspired by the Banach-Alaoglu theorem: for a bounded subset $A$ of $X$ we set:
\[ \textrm{wk}_{X}(A)=\widehat{\textrm{d}}\big(\overline{A}^{\sigma(X^{**},X^{*})},X\big). \]

It is a direct consequence of the Banach--Alaoglu theorem that $A$ is relatively weakly compact if and only if $\textrm{wk}_{X}(A)=0$. The second one is inspired by the Eberlein--\v{S}mulyan theorem: for a bounded subset $A$ of $X$ we set:
\[ \operatorname{wck}_{X}(A)=\sup\{\textrm{d}(\textrm{clust}_{X^{**}}((x_{n})_{n=1}^\infty),X)\colon (x_{n})_{n=1}^\infty \text{ is a sequence in }A\}, \]
where $\textrm{clust}_{X^{**}}((x_{n})_{n=1}^\infty)$ is the set of all weak$^{*}$-cluster points of $(x_{n})_{n=1}^\infty$ in $X^{**}$.

It follows easily from the Eberlein--\v{S}mulyan theorem that $\operatorname{wck}_{X}(A)=0$ whenever $A$ is relatively weakly compact. The converse follows from the quantitative version of the Eberlein--\v{S}mulyan theorem as proved in \cite{AC1}. It follows from \cite[Theorem 2.3]{AC} that
\begin{equation}\label{9}
    \operatorname{wck}_{X}(A)\leqslant \textrm{wk}_{X}(A)\leqslant 2\operatorname{wck}_{X}(A).
\end{equation}

It should be pointed out that $\operatorname{wck}_{X}(B_{X})=1$ for every non-reflexive Banach space $X$. This follows for example from \cite[Theorem 1]{GHP}. By \eqref{9}, this also holds for $\operatorname{wk}_{X}(B_{X})$. Hence
\begin{equation}\label{15}
    \operatorname{wck}_{X}(B_{X})=\operatorname{wck}_{X^{*}}(B_{X^{*}}), \quad \operatorname{wk}_{X}(B_{X})=\operatorname{wk}_{X^{*}}(B_{X^{*}}).
\end{equation}
If $Y$ is a subspace of a Banach space $X$ and $A$ is a bounded subset of $Y$, it is clear that $A$ is relatively weakly compact in $Y$ if and only if it is relatively weakly compact in $X$. This elementary fact has the following quantitative version:
\begin{equation}\label{14}
    \operatorname{wk}_{X}(A)\leq \operatorname{wk}_{Y}(A)\leq 2\operatorname{wk}_{X}(A), \quad \operatorname{wck}_{X}(A)\leq \operatorname{wck}_{Y}(A)\leq 2\operatorname{wck}_{X}(A).
\end{equation}
In both cases the first inequality is trivial, whereas the second one follows from \cite[Lemma 11]{Gra}. It is worth mentioning that the factor 2 in the right-hand side of the inequalities \eqref{14} is optimal. Indeed, let $Y=c_{0}, X=\ell_{\infty}$, and $A$ be the summing basis of $c_{0}$. A~standard argument shows that $\textrm{wck}_{Y}(A)=\textrm{wk}_{Y}(A)=1$ and $\textrm{wck}_{X}(A)=\textrm{wk}_{X}(A)=1/2$.

Let $A$ be a subset of a Banach space $X$, we set
\[
    \textrm{sep}(A)=\inf\{\varepsilon>0\colon A\subseteq C+\varepsilon B_{X}, C\subseteq X\text{ countable}\}.
\]
Clearly, $A$ is separable if and only if $\textrm{sep}(A)=0$.

\begin{lemma}\label{2.1.1}
Let $X$ be a Banach space. Then $\operatorname{sep}(B_{X})\in \{0,1\}.$  In particular,
\[
    \operatorname{sep}(B_{X})\leq \operatorname{sep}(B_{X^{*}}).
\]
\end{lemma}

\begin{proof}
Suppose that $\operatorname{sep}(B_{X})<1$. Then there exist $\varepsilon<1$ and a countable subset $C_{1}$ of $X$ so that $B_{X}\subseteq C_{1}+\varepsilon B_{X}$.
This implies that
\[
    B_{X}\subseteq C_{1}+\varepsilon(C_{1}+\varepsilon B_{X})=C_{2}+\varepsilon^{2}B_{X},
\]
where $C_{2}=C_{1}+\varepsilon C_{1}$ is countable.

Inductively, we get a sequence of countable sets $(C_{n})_{n}$ so that $B_{X}\subseteq C_{n}+\varepsilon^{n}B_{X}$ for all $n$. Hence $\operatorname{sep}(B_{X})\leq \varepsilon^{n}$ for all $n$ and so $\operatorname{sep}(B_{X})=0$.\end{proof}

\subsection{Measures of isomorphisms between Banach spaces}

The first-named author \cite{Chen} introduced a quantity measuring how well a Banach space can be embedded into another Banach space. More precisely, let $X,Y$ be Banach spaces. If $X$ and $Y$ are isomorphic, we set
\[
    \alpha_{Y}(X)=\sup\{\|T^{-1}\|^{-1}\colon T\colon Y\rightarrow X\text{ is an isomorphism with} \|T\|\leq 1\}.
\]

If there is no isomorphism from $Y$ into $X$, we set $\alpha_{Y}(X)=0$. We have $\alpha_{Y}(X)=1$ if and only if $X$ contains almost isometric copies of $Y$. We also need a quantity \cite{Chen} measuring how well a Banach space is from being isomorphic to a complemented subspace of another Banach space.

For a pair of Banach spaces $X,Y$ we set
\[ \beta_{Y}(X)=\sup\{(\|A\|\|B\|)^{-1}\colon A\colon X\rightarrow Y, B\colon Y\rightarrow X\text{ are operators such that }AB=I_{Y}\}. \]
If there are no such operators $A,B$, we set $\beta_{Y}(X)=0$. Clearly, $\beta_{Y}(X)=1$ if and only if for every $\varepsilon>0$, there exists a subspace $M$ of $X$ so that $M$ is $(1+\varepsilon)$-isomorphic to $Y$ and $M$ is $(1+\varepsilon)$-complemented in $X$. Let us collect some elementary observations that we shall later employ.

\begin{lemma}\label{2.1.2}
Let $X,Y$ be Banach spaces. Then
\begin{romanenumerate}
\item\label{i1:2.1.2} $\beta_{Y}(X)\leq \beta_{Y^{*}}(X^{*})\leq\alpha_{Y^{*}}(X^{*})$.
\item\label{i2:2.1.2} $\alpha_{c_{0}}(X)=\beta_{c_{0}}(X)$ if $X$ is separable.
\item\label{i3:2.1.2} $\alpha_{\ell_{1}}(X)\leq \operatorname{wck}_{X}(B_{X})$.
\end{romanenumerate}
\end{lemma}
\begin{proof}\eqref{i1:2.1.2} is straightforward, whereas \eqref{i2:2.1.2} follows from \cite[Theorem 6]{DRT}. \eqref{i3:2.1.2} follows from \cite[Lemma 5]{KPS}.\end{proof}

\section{Quantification of shrinking bases}\label{sec:shrinking}
\begin{definition}
Let $X$ be a Banach space with a basis $(x_{n})_{n=1}^\infty$. We set
\[
    \textrm{sh}_{X}((x_{n})_{n=1}^\infty)=\sup_{x^{*}\in B_{X^{*}}}\limsup_{n}\|x^{*}\|_{n}.
\]
\end{definition}
According to the definition above, $(x_{n})_{n=1}^\infty$ is shrinking if and only if $\textrm{sh}((x_{n})_{n=1}^\infty)=0$.

If $(x_{n})_{n}$ is a basic sequence, we are able to define $\textrm{sh}_{[x_{n}:n\in \mathbb{N}]}((x_{n})_{n})$ naturally. For the sake of simplicity, we'll omit the subscripts $X$ in $\textrm{sh}_{X}((x_{n})_{n})$ for a basis $(x_{n})_{n}$ and $[x_{n}:n\in \mathbb{N}]$ in $\textrm{sh}_{[x_{n}:n\in \mathbb{N}]}((x_{n})_{n})$ for a basic sequence $(x_{n})_{n}$.

\begin{example}${}$
\begin{enumerate}
\item\label{1:ex1} $\textrm{sh}((e^{*}_{n})_{n=1}^\infty)=1$, where $(e^{*}_{n})_{n=1}^\infty$ is the unit vector basis of $\ell_{1}$.
\item\label{1:ex2} $\textrm{sh}((s_{n})_{n=1}^\infty)=1$, where $(s_{n})_{n=1}^\infty$ is the summing basis of $c_{0}$.
\item\label{1:ex3} $\textrm{sh}((\sum_{i=1}^{n}e_{i})_{n=1}^\infty)=1$, where $(e_{n})_{n=1}^\infty$ is the unit vector basis of the James space $\mathcal{J}$.
\end{enumerate}
\end{example}
\begin{proof}
We only prove \eqref{1:ex3}. Let $(f_{n})_{n=1}^\infty$ be the biorthogonal functionals associated to $(e_{n})_{n=1}^\infty$. Since $(e_{n})_{n=1}^\infty$ is monotone, we get $\|f_{1}\|=1$.
Clearly, $\|\sum_{i=1}^{n}e_{i}\|_{\mathcal{J}}=1$ for all $n$. Thus $\|f_{1}|_{[\sum_{i=1}^{k}e_{i}\colon k>n]}\|=1$ for all $n$.
\end{proof}

\begin{theorem}\label{2.1}
Let $X$ be a Banach space with a basis $(x_{n})_{n=1}^\infty$. Then
\[\operatorname{sh}((x_{n})_{n=1}^\infty)\leq \widehat{\operatorname{d}}(B_{X^{*}},V)\leq (K+1)\operatorname{sh}((x_{n})_{n=1}^\infty).\]
\end{theorem}
\begin{proof}
Let $0<c<\operatorname{sh}((x_{n})_{n=1}^\infty)$. By (\ref{1}), there exist $x^{*}_{0}\in B_{X^{*}}$ and a strictly increasing sequence $(k_{n})_{n}$ of positive integers so that $\textrm{d}(x^{*}_{0},[x^{*}_{i}\colon i\leq k_{n}])>c$ for all $n$. We \emph{claim} that $\textrm{d}(x^{*}_{0},V)\geq c$. Indeed, given $x^{*}\in V$ and $\varepsilon>0$. We choose $y^{*}\in [x^{*}_{i}\colon i\leq N]$ for some $N$ so that $\|x^{*}-y^{*}\|<\varepsilon$. Hence we get
\[
    \|x^{*}_{0}-x^{*}\|\geq \|x^{*}_{0}-y^{*}\|-\|y^{*}-x^{*}\|>c-\varepsilon.
\]
Letting $\varepsilon\rightarrow 0$, we get the claim. Since $c$ is arbitrary, we arrive at the first inequality. It remains to verify the second inequality.

For this, fix an arbitrary $0<c<\widehat{\operatorname{d}}(B_{X^{*}},V)$. We choose $x^{*}_{0}\in B_{X^{*}}$ with $\textrm{d}(x^{*}_{0},V)>c$. For each $n$, we get
\begin{align*}
c&\leq \|x^{*}_{0}-\sum_{i=1}^{n}\langle x^{*}_{0},x_{i}\rangle x^{*}_{i}\|\\
&=\sup_{x\in B_{X}}|\langle x^{*}_{0},x\rangle-\sum_{i=1}^{n}\langle x^{*}_{0},x_{i}\rangle \langle x^{*}_{i},x\rangle|\\
&=\sup_{x\in B_{X}}|\langle x^{*}_{0},x-\sum_{i=1}^{n}\langle x^{*}_{i},x\rangle x_{i}\rangle|\\
&=\sup_{x\in B_{X}}|\langle x^{*}_{0},\sum_{i=n+1}^{\infty}\langle x^{*}_{i},x\rangle x_{i}\rangle|\\
&\leq (K+1)\|x^{*}_{0}|_{[x_{i}\colon i>n]}\|.\\
\end{align*}
This implies that $c\leq (K+1)\limsup_{n}\|x^{*}_{0}\|_{n}$.
As $c$ is arbitrary, the proof is complete.\end{proof}

\begin{theorem}\label{2.2}
Let $X$ be a Banach space with a basis $(x_{n})_{n=1}^\infty$.
Then
\[\alpha_{\ell_{1}}(X)\leq \operatorname{sep}(B_{X^*})\leq \widehat{\operatorname{d}}(B_{X^{*}},V).\]
\end{theorem}
\begin{proof}
Let $0<c<\alpha_{\ell_{1}}(X)$. Then there exists an operator $T\colon \ell_{1}\rightarrow X$ such that
\[
    c\|z\|\leq \|Tz\|\leq \|z\|\quad (z\in \ell_1).
\]
This implies that $T^{*}B_{X^{*}}\supseteq cB_{\ell_{\infty}}\supseteq cA$, where $A=\{(\theta_{n})_{n=1}^\infty\colon |\theta_{n}|= 1\; (n\in \mathbb{N})\}$. It is easy to see that $\textrm{sep}(A)=1$. Hence we get \[c\leq \textrm{sep}(T^{*}B_{X^{*}})\leq \textrm{sep}(B_{X^{*}}).\]
By the arbitrariness of $c$, we get the first inequality.

Let $c>\widehat{\operatorname{d}}(B_{X^{*}},V)$ and $\varepsilon>0$. Moreover, let $\mathcal Q$ be a countable, dense subset of $\mathbb R$ and let $C=\{\sum_{i=1}^{n}r_{i}x^{*}_{i}\colon n\in \mathbb{N},r_{1},r_{2},\ldots,r_{n}\in \mathcal{Q}\}$. For $x^{*}\in B_{X^{*}}$ we choose $y^{*}\in V$ with $\|x^{*}-y^{*}\|<c$ and $z^{*}\in C$ with $\|y^{*}-z^{*}\|<\varepsilon.$ Hence $\|x^{*}-z^{*}\|<c+\varepsilon$. This means that $\textrm{sep}(B_{X^{*}})\leq c+\varepsilon$. As $c$ and $\varepsilon$ were arbitrary, the proof is complete.\end{proof}

\begin{theorem}\label{2.3}
Let $X$ be a Banach space with an unconditional basis $(x_{n})_{n=1}^\infty$. Then
\[
    \operatorname{sh}((x_{n})_{n=1}^\infty)\leq K_{u}\alpha_{\ell_{1}}(X).
\]
\end{theorem}
\begin{proof}
Let $0<c<\operatorname{sh}((x_{n})_{n=1}^\infty)$. Then there exist $x^{*}_{0}\in B_{X^{*}}$ and a block basic sequence $(u_{n})_{n}$ with respect to $(x_{n})_{n=1}^\infty$ so that $\|u_{n}\|\leq 1$ and $\langle x^{*}_{0},u_{n}\rangle >c$ for all $n$. Then, for each $m$ and every choice of scalars $a_{1},a_{2},\ldots,a_{m}$, we get
\[
    K_{u}\|\sum_{n=1}^{m}a_{n}u_{n}\|\geq \|\sum_{n=1}^{m}|a_{n}|u_{n}\|\geq \sum_{n=1}^{m}|a_{n}|\langle x^{*}_{0},u_{n}\rangle\geq c\sum_{n=1}^{m}|a_{n}|.
\]
Let us define an operator $T\colon \ell_{1}\rightarrow X$ by $Te_{n}=u_{n}$ $(n\in \mathbb{N})$. It is easy to see that $\|T\|\leq 1$ and $\|T^{-1}\|\leq \frac{K_{u}}{c}$.
This implies that $\alpha_{\ell_{1}}(X)\geq \frac{c}{K_{u}}$. As $c$ was arbitrary, the proof is complete.\end{proof}

\section{Quantifications of bounded completeness}\label{sect:boundedly}
Let $(x_{n})_{n=1}^\infty$ be a bounded sequence in a Banach space $X$. We set
\[
    \textrm{ca}((x_{n})_{n=1}^\infty)=\inf_{n}\sup_{k,l\geq n}\|x_{k}-x_{l}\|.
\]
Then $(x_{n})_{n=1}^\infty$ is norm-Cauchy if and only if $\textrm{ca}((x_{n})_{n=1}^\infty)=0$.

\begin{definition}
Let $(x_{n})_{n=1}^\infty$ be a basis for a Banach space $X$. We set
\[
    \textrm{bc}_{1}((x_{n})_{n=1}^\infty)=\sup\Big\{\textrm{ca}((\sum_{i=1}^{n}a_{i}x_{i})_{n=1}^\infty)\colon
(\sum_{i=1}^{n}a_{i}x_{i})_{n=1}^\infty\subseteq B_{X}\Big\}.
\]
\end{definition}
Clearly, $(x_{n})_{n=1}^\infty$ is boundedly complete if and only if $\textrm{bc}_{1}((x_{n})_{n=1}^\infty)=0$.

\begin{example}${}$
\begin{enumerate}
\item $\operatorname{bc}_{1}((e_{n})_{n})=1$, where $(e_{n})_{n}$ is the unit vector basis of $c_{0}$.
\item $\operatorname{bc}_{1}((s_{n})_{n})=1$, where $(s_{n})_{n}$ is the summing basis of $c_{0}$.
\item $\operatorname{bc}_{1}((e_{n})_{n=0}^{\infty})=2$, where $(e_{n})_{n=0}^{\infty}$ is the unit vector basis of $c$ ($e_{0}=(1,1,1,\ldots)$).
\item $\operatorname{bc}_{1}((e_{n})_{n})=1$, where $(e_{n})_{n}$ is the unit vector basis of the James space $\mathcal{J}$.
\end{enumerate}
\end{example}
\begin{proof}
Example (1) is straightforward. For (2), note that $\|\sum_{i=1}^{n}a_{i}s_{i}\|=\max_{1\leq k\leq n}|\sum_{i=k}^{n}a_{i}|$ for all $n$ and all scalars $a_{1},a_{2},\ldots,a_{n}$. Given $(\sum_{i=1}^{n}a_{i}s_{i})_{n}\subseteq B_{c_{0}}$. Then $|\sum_{i=k}^{n}a_{i}|\leq 1$ for all $n$ and all $k\leq n$. This implies that
\[
    \|\sum_{i=n+1}^{n+m}a_{i}s_{i}\|=\max_{n+1\leq k\leq n+m}|\sum_{i=k}^{n+m}a_{i}|\leq 1\quad (m,n\in\mathbb N).
\]
Hence $\textrm{ca}((\sum_{i=1}^{n}a_{i}s_{i})_{n})\leq 1$ and so $\operatorname{bc}_{1}((s_{n})_{n})\leq 1$.

On the other hand, observe that
\[
    \|\sum_{i=1}^{n}(-1)^{i}s_{i}\|=\max_{1\leq k\leq n}|\sum_{i=k}^{n}(-1)^{i}|\leq 1\quad (n\in\mathbb N).
\]
However, for each $n$, we have
\[
    \|\sum_{i=1}^{n+2n-1}(-1)^{i}s_{i}-\sum_{i=1}^{n}(-1)^{i}s_{i}\|=\max_{n+1\leq k\leq n+2n-1}|\sum_{i=k}^{n+2n-1}(-1)^{i}|
\geq |\sum_{i=n+1}^{n+2n-1}(-1)^{i}|=1.
\]
This implies that $\textrm{ca}((\sum_{i=1}^{n}(-1)^{i}s_{i})_{n=1}^\infty)\geq 1$. Hence $\operatorname{bc}_{1}((s_{n})_{n})\geq 1$.

\noindent (3). Observe that $\|\sum_{i=0}^{n}a_{i}e_{i}\|=\max(|a_{0}|,|a_{0}+a_{1}|,\ldots,|a_{0}+a_{n}|)$ for all $n$ and all scalars $a_{0},a_{1},\ldots,a_{n}$. We take $a_{0}=-1,a_{n}=2$ $(n=1,2,\ldots).$ Then $\|\sum_{i=0}^{n}a_{i}e_{i}\|=1$ for all $n$. However, for all $n,k$, we get
\[
    \Big\|\sum_{i=0}^{n+k}a_{i}e_{i}-\sum_{i=0}^{n}a_{i}e_{i}\Big\|=\sup_{n+1\leq i\leq n+k}|a_{i}|=2.
\]
This means that $\textrm{ca}((\sum_{i=0}^{n}a_{i}e_{i})_{n=1}^\infty)=2$. Hence $\operatorname{bc}_{1}((e_{n})_{n=0}^{\infty})=2$.

\noindent (4). Note that $\|\sum_{i=n}^{m}e_{i}\|_{\mathcal{J}}=1$ for all $n,m$. This implies that $\operatorname{bc}_{1}((e_{n})_{n})\geq 1$. For the second inequality, assume that $\|\sum_{i=1}^{n}a_{i}e_{i}\|_{\mathcal{J}}\leq 1$ for all $n$.
Suppose $p_{1}<p_{2}<\ldots<p_{n+1}$. Since $\|\sum_{i=1}^{p_{n+1}}a_{i}e_{i}\|_{\mathcal{J}}\leq 1$, we get
$\sum_{i=1}^{n}(a_{p_{i}}-a_{p_{i+1}})^{2}+a_{p_{n+1}}^{2}+a_{p_{1}}^{2}\leq 2.$
This implies that $\|\sum_{i=n}^{m}a_{i}e_{i}\|_{\mathcal{J}}\leq 1$ for all $n,m$ and so $\operatorname{bc}_{1}((e_{n})_{n})\leq 1$.
\end{proof}

\begin{definition}
Let $X$ be a Banach space with a basis $(x_{n})_{n=1}^\infty$. We set
\begin{itemize}
    \item $\operatorname{bc}_{2}((x_{n})_{n=1}^\infty)=\sup_{\varphi\in B_{V^{*}}}\operatorname{ca}((\sum_{i=1}^{n}\langle \varphi,x^{*}_{i}\rangle x_{i})_{n=1}^\infty)$
    \item $\operatorname{bc}_{3}((x_{n})_{n=1}^\infty)=\sup_{x^{**}\in B_{X^{**}}}\operatorname{ca}((\sum_{i=1}^{n}\langle x^{**},x^{*}_{i}\rangle x_{i})_{n=1}^\infty).$
\end{itemize}
\end{definition}

\begin{theorem}\label{4.4}
Let $X$ be a Banach space with a basis $(x_{n})_{n=1}^\infty$. Then
\[\operatorname{bc}_{1}((x_{n})_{n=1}^\infty)\leq \operatorname{bc}_{2}((x_{n})_{n=1}^\infty)\leq \operatorname{bc}_{3}((x_{n})_{n=1}^\infty)\leq K\operatorname{bc}_{1}((x_{n})_{n=1}^\infty).\]
\end{theorem}
\begin{proof}
Claim 1: $\operatorname{bc}_{1}((x_{n})_{n=1}^\infty)\leq \operatorname{bc}_{2}((x_{n})_{n=1}^\infty)$.

Given a sequence of scalars $a_1,\ldots, a_n$ so that $\|\sum_{i=1}^{n}a_{i}x_{i}\|\leq 1$ for all $n$, we have
\[
    |\sum_{i=1}^{n}b_{i}a_{i}|=|\langle \sum_{j=1}^{n}b_{j}x^{*}_{j},\sum_{i=1}^{n}a_{i}x_{i}\rangle|\leq \|\sum_{j=1}^{n}b_{j}x^{*}_{j}\|
\]
for all $n$ and all scalars $b_{1},b_{2},\ldots,b_{n}$. By Helly's theorem, there exists $\varphi\in V^{*}$ so that $\|\varphi\|\leq 1$ and $\langle\varphi,x^{*}_{n}\rangle=a_{n}$ for all $n$. This proves Claim 1.\smallskip

\noindent Claim 2: $\operatorname{bc}_{2}((x_{n})_{n=1}^\infty)\leq \operatorname{bc}_{3}((x_{n})_{n=1}^\infty)$.

Let $\varphi\in B_{V^{*}}$ and let $\varepsilon>0$. We choose $x^{**}\in X^{**}$ so that $x^{**}|_{V}=\varphi$ and $\|x^{**}\|\leq 1+\varepsilon$. Hence
\[
    \operatorname{ca}((\sum_{i=1}^{n}\langle \varphi,x^{*}_{i}\rangle x_{i})_{n})=(1+\varepsilon)\operatorname{ca}((\sum_{i=1}^{n}\langle \frac{x^{**}}{1+\varepsilon},x^{*}_{i}\rangle x_{i})_{n=1}^\infty)\leq (1+\varepsilon)\operatorname{bc}_{3}((x_{n})_{n=1}^\infty).
\]
Thus
\[
    \operatorname{bc}_{2}((x_{n})_{n=1}^\infty)\leq (1+\varepsilon)\operatorname{bc}_{3}((x_{n})_{n=1}^\infty).
\]
Letting $\varepsilon\rightarrow 0$, we arrive at the sought conclusion.\smallskip

\noindent Claim 3: $\operatorname{bc}_{3}((x_{n})_{n=1}^\infty)\leq K\operatorname{bc}_{1}((x_{n})_{n=1}^\infty)$. For this, observe that $P_{n}^{**}x^{**}=\sum_{i=1}^{n}\langle x^{**},x^{*}_{i}\rangle x_{i}$ for each $x^{**}\in X^{**}$.
Hence for every $x^{**}\in B_{X^{**}}$, $\|\sum_{i=1}^{n}\langle x^{**},x^{*}_{i}\rangle x_{i}\|\leq K$. This implies that
\[\operatorname{ca}((\sum_{i=1}^{n}\langle x^{**},x^{*}_{i}\rangle x_{i})_{n=1}^\infty)\leq K\operatorname{bc}_{1}((x_{n})_{n=1}^\infty),\]
which completes the proof.\end{proof}

\begin{theorem}\label{4.5}
Let $X$ be a Banach space with a basis $(x_{n})_{n=1}^\infty$. Then
\[
    \operatorname{sh}((x^{*}_{n})_{n=1}^{\infty})\leq \operatorname{bc}_{2}((x_{n})_{n=1}^\infty)\leq 2K^{2}\operatorname{sh}((x^{*}_{n})_{n=1}^{\infty}).
\]
\end{theorem}
\begin{proof}Claim 1: for every $\varphi\in V^{*}$ we have $\limsup_{n}\|\varphi|_{[x^{*}_{i}\colon i>n]}\|\leq \operatorname{ca}((\sum_{i=1}^{n}\langle \varphi,x^{*}_{i}\rangle x_{i})_{n=1}^\infty)$.

Let $0<c<\limsup_{n}\|\varphi|_{[x^{*}_{i}\colon i>n]}\|$. Then there exists a block basic sequence $(f_{n})_{n=1}^\infty$ with respect to $(x^{*}_{n})_{n=1}^{\infty}$ so that $\|f_{n}\|\leq 1$ and $|\langle \varphi,f_{n}\rangle|>c$ for all $n$. Let us write $f_{n}=\sum_{i=k_{n-1}+1}^{k_{n}}\langle f_{n},x_{i}\rangle x^{*}_{i}$. Then we get
\[
    \Big\|\sum_{i=k_{n-1}+1}^{k_{n}}\langle \varphi,x_{i}^{*}\rangle x_{i}\Big\|\geq \Big|\big\langle f_{n},\sum_{i=k_{n-1}+1}^{k_{n}}\langle \varphi,x_{i}^{*}\rangle x_{i}\big\rangle\Big|=|\langle \varphi,f_{n}\rangle|>c \quad (n\in \mathbb N).
\]
This implies that $\operatorname{ca}((\sum_{i=1}^{n}\langle \varphi,x^{*}_{i}\rangle x_{i})_{n})>c$ and so Claim 1 is established.\smallskip

\noindent Claim 2: $\operatorname{ca}((\sum_{i=1}^{n}\langle \varphi,x^{*}_{i}\rangle x_{i})_{n=1}^\infty)\leq 2K^{2}\limsup_{n}\|\varphi|_{[x^{*}_{i}\colon i>n]}\|$ for every $\varphi\in V^{*}$.

Let $0<c<\operatorname{ca}((\sum_{i=1}^{n}\langle \varphi,x^{*}_{i}\rangle x_{i})_{n=1}^\infty)$. Then there exists a strictly increasing sequence of positive integers $(k_{n})_{n}$ so that $\|\sum_{i=k_{2n-1}+1}^{k_{2n}}\langle \varphi,x_{i}^{*}\rangle x_{i}\|>c$ for all $n$. It follows from (\ref{2}) that
\[
    \|\sum_{i=k_{2n-1}+1}^{k_{2n}}\langle \varphi,x_{i}^{*}\rangle jx_{i}\|>\frac{c}{K}\quad (n\in \mathbb N).
\]
For each $n$ we choose $f_{n}\in B_{V}$ so that
\[
    \Big|\big\langle \varphi,\sum_{i=k_{2n-1}+1}^{k_{2n}}\langle jx_{i},f_{n}\rangle x^{*}_{i}\big\rangle\Big|=\Big|\sum_{i=k_{2n-1}+1}^{k_{2n}}\langle \varphi,x_{i}^{*}\rangle \langle jx_{i},f_{n}\rangle\Big|>\frac{c}{K}.
\]
Since $\|\sum_{i=k_{2n-1}+1}^{k_{2n}}\langle jx_{i},f_{n}\rangle x^{*}_{i}\|\leq 2K$ for each $n$, we get $\|\varphi|_{[x^{*}_{i}\colon i>k_{2n-1}]}\|\geq \frac{c}{2K^{2}}$ and so $\limsup_{n}\|\varphi|_{[x^{*}_{i}\colon i>n]}\|\geq \frac{c}{2K^{2}}$. As $c$ was arbitrary, the proof is complete.\end{proof}

\begin{theorem}\label{4.6}
Let $X$ be a Banach space with a basis $(x_{n})_{n=1}^\infty$. Then
\[
    \frac{1}{2K}\operatorname{bc}_{2}((x^{*}_{n})_{n=1}^{\infty})\leq \operatorname{sh}((x_{n})_{n=1}^\infty)\leq K\operatorname{bc}_{2}((x^{*}_{n})_{n=1}^{\infty}).
\]
\end{theorem}
\begin{proof}
Let $0<c<\operatorname{bc}_{2}((x^{*}_{n})_{n=1}^{\infty})$. Take $f\in B_{W^{*}}$ so that $\textrm{ca}((\sum_{i=1}^{n}\langle f,jx_{i}\rangle x^{*}_{i})_{n=1}^\infty)>c$. We choose a strictly increasing sequence of integers $(k_{n})_{n=1}^\infty$ so that $\|\sum_{i=k_{2n-1}+1}^{k_{2n}}\langle f,jx_{i}\rangle x^{*}_{i}\|>c$ ($n\in \mathbb N$). For each $n$, we take $y_{n}\in B_{X}$ with $|\sum_{i=k_{2n-1}+1}^{k_{2n}}\langle f,jx_{i}\rangle \langle x^{*}_{i},y_{n}\rangle|>c$. Define $x^{*}\in X^{*}$ by $\langle x^{*},x\rangle=\langle f,jx\rangle (x\in X)$. Then $x^{*}\in B_{X^{*}}$ and
\[
    \Big|\langle x^{*},\sum_{i=k_{2n-1}+1}^{k_{2n}}\langle x^{*}_{i},y_{n}\rangle x_{i}\rangle\Big|=\Big|\sum_{i=k_{2n-1}+1}^{k_{2n}}\langle x^{*},x_{i}\rangle
\langle x^{*}_{i},y_{n}\rangle\Big|>c\quad (n\in \mathbb N).
\]
Note that
\[
    \|\sum_{i=k_{2n-1}+1}^{k_{2n}}\langle x^{*}_{i},y_{n}\rangle x_{i}\|\leq 2K \quad (n\in \mathbb N).
\]
Consequently,
\[\|x^{*}|_{[x_{i}\colon i>k_{2n-1}]}\|\geq \frac{c}{2K} \quad (n\in \mathbb N).
\]
This implies
\[
    \limsup_{n\to\infty}\|x^{*}|_{[x_{i}\colon i>n]}\|\geq \frac{c}{2K}
\]
so $\operatorname{sh}((x_{n})_{n=1}^\infty)\geq \frac{c}{2K}$. As $c$ was arbitrary, we arrive at the former inequality.

As to the latter one, let $0<c<\operatorname{sh}((x_{n})_{n=1}^\infty)$. Then there exists $x^{*}\in B_{X^{*}}$ so that $\limsup_{n}\|x^{*}|_{[x_{i}\colon i>n]}\|>c$.
We choose a block basic sequence $(u_{n})_{n=1}^\infty$ with respect to $(x_{n})_{n=1}^\infty$ so that $\|u_{n}\|\leq 1$ and $|\langle x^{*},u_{n}\rangle|>c$ for all $n$. Write each $u_{n}=\sum_{i=k_{n-1}+1}^{k_{n}}\langle x^{*}_{i},u_{n}\rangle x_{i}$.
We define $f\in W^{*}$ by $\langle f,jx\rangle=\langle x^{*},x\rangle$ ($x\in X$). By (\ref{2}), we get $\|f\|\leq K$. Moreover, we get
\[
    c<\Big|\sum_{i=k_{n-1}+1}^{k_{n}}\langle x^{*}_{i},u_{n}\rangle \langle x^{*},x_{i}\rangle\Big|
    =\Big|\langle \sum_{i=k_{n-1}+1}^{k_{n}}\langle f,jx_{i}\rangle x^{*}_{i},u_{n}\rangle\Big|
    \leq \Big\|\sum_{i=k_{n-1}+1}^{k_{n}}\langle f,jx_{i}\rangle x^{*}_{i}\Big\|\quad (n\in \mathbb N).
\]
Thus $\textrm{ca}((\sum_{i=1}^{n}\langle f,jx_{i}\rangle x^{*}_{i})_{n=1}^\infty)\geq c$ and then $\operatorname{bc}_{2}((x^{*}_{n})_{n=1}^{\infty})\geq \frac{c}{K}$. By the arbitrariness of $c$, we complete the proof. \end{proof}

\begin{theorem}\label{4.7}
Let $X$ be a Banach space with an unconditional basis $(x_{n})_{n=1}^\infty$. Then
\[
    \frac{1}{K_{u}}\alpha_{c_{0}}(X)\leq \operatorname{bc}_{1}((x_{n})_{n=1}^\infty)\leq K_{u}^{3}\alpha_{c_{0}}(X).
\]
\end{theorem}
\begin{proof}
Let $0<c<\alpha_{c_{0}}(X)$. Then there exists a sequence $(y_{n})_{n=1}^\infty$ in $X$ so that
\begin{equation}\label{3}
c\cdot\max_{1\leq i\leq n}|t_{i}|\leq \|\sum_{i=1}^{n}t_{i}y_{i}\|\leq \max_{1\leq i\leq n}|t_{i}|\quad (n\in\mathbb N,\, t_{1},t_{2},\ldots,t_{n}\in \mathbb K).
\end{equation}

Let $\varepsilon>0$. It follows from a quantitative version of the Bessaga--Pe{\l}czy\'{n}ski Selection Principle (\cite[Lemma 2.3]{Chen}) that there exist a subsequence $(y_{k_{n}})_{n=1}^\infty$ of $(y_{n})_{n=1}^\infty$ and a block basic sequence $(u_{n})_{n=1}^\infty$ with respect to $(x_{n})_{n=1}^\infty$ so that
\begin{equation}\label{4}
(1-\varepsilon)\|\sum_{i=1}^{n}t_{i}u_{i}\|\leq \|\sum_{i=1}^{n}t_{i}y_{k_{i}}\|\leq (1+\varepsilon)\|\sum_{i=1}^{n}t_{i}u_{i}\|\quad (n\in \mathbb N,\, t_{1},t_{2},\ldots,t_{n}\in \mathbb K)
\end{equation}
Combining (\ref{3}) and (\ref{4}), we get
\begin{equation}\label{5}
\frac{c}{1+\varepsilon}\max_{1\leq i\leq n}|t_{i}|\leq\|\sum_{i=1}^{n}t_{i}u_{i}\|\leq \frac{1}{1-\varepsilon}\max_{1\leq i\leq n}|t_{i}|\quad (n\in \mathbb N,\, t_{1},t_{2},\ldots,t_{n}\in \mathbb K).
\end{equation}

Write $u_{n}=\sum_{i=k_{n-1}+1}^{k_{n}}a_{i}x_{i}$. By \cite[Proposition 1.c.7]{LT} and (\ref{5}), we get
\begin{equation}\label{6}
\|\sum_{j=1}^{n}a_{j}x_{j}\|\leq K_{u}\|\sum_{j=1}^{n}u_{j}\|\leq \frac{K_{u}}{1-\varepsilon} \quad (n\in \mathbb N)
\end{equation}
and
\begin{equation}\label{7}
\|\sum_{i=k_{n-1}+1}^{k_{n}}a_{i}x_{i}\|=\|u_{n}\|\geq \frac{c}{1+\varepsilon} \quad (n\in \mathbb N)
\end{equation}
Using \eqref{6} and \eqref{7}, we arrive at
\[
    \textrm{bc}_{1}((x_{n})_{n=1}^\infty)\geq \frac{c}{1+\varepsilon}\frac{1-\varepsilon}{K_{u}}.
\]
Letting $\varepsilon\rightarrow 0$, we get $\textrm{bc}_{1}((x_{n})_{n=1}^\infty)\geq \frac{c}{K_{u}}$. Since $c$ is arbitrary, we arrive at the former inequality.

It remains to verify the latter one. Let $0<c<\textrm{bc}_{1}((x_{n})_{n=1}^\infty)$. Then there exists a scalar sequence $(a_{n})_{n=1}^\infty$ so that $\|\sum_{i=1}^{n}a_{i}x_{i}\|\leq 1$ for all $n$ and $\textrm{ca}((\sum_{i=1}^{n}a_{i}x_{i})_{n=1}^\infty)>c$. We choose $k_{1}<k_{2}<\ldots<k_{n}<\ldots$ so that $\|\sum_{i=k_{2n-1}+1}^{k_{2n}}a_{i}x_{i}\|>c$ for all $n$. Let $u_{n}=\sum_{i=k_{2n-1}+1}^{k_{2n}}a_{i}x_{i}$. Then $\|u_{n}\|>c$ for every $n$.
Given a finite choice of scalars $(t_{n})_{n=1}^{m}$. Appealing again to \cite[Proposition 1.c.7]{LT}, we get
\begin{align*}
    \|\sum_{n=1}^{m}t_{n}u_{n}\|&=\|\sum_{n=1}^{m}\sum_{i=k_{2n-1}+1}^{k_{2n}}t_{n}a_{i}x_{i}\|\\
    &\leq K_{u}\max_{1\leq n\leq m}|t_{n}|\|\sum_{n=1}^{m}u_{n}\|\\
    &\leq K_{u}\max_{1\leq n\leq m}|t_{n}|K_{u}\|\sum_{i=1}^{k_{2m}}a_{i}x_{i}\|\\
    &\leq K_{u}^{2}\max_{1\leq n\leq m}|t_{n}|.\\
\end{align*}
On the other hand, for each $1\leq n\leq m$, we get
\[
    c|t_{n}|\leq \|t_{n}u_{n}\|\leq K_{u}\|\sum_{n=1}^{m}t_{n}u_{n}\|.
\]
Hence
\[
    c\cdot \max_{1\leq n\leq m}|t_{n}|\leq K_{u}\|\sum_{n=1}^{m}t_{n}u_{n}\|.
\]
In conclusion,
\begin{equation}\label{8}
\frac{c}{K_{u}}\max_{1\leq n\leq m}|t_{n}|\leq \|\sum_{n=1}^{m}t_{n}u_{n}\|\leq K_{u}^{2}\max_{1\leq n\leq m}|t_{n}|
\end{equation}
We define an operator $T\colon c_{0}\rightarrow X$ by $e_{n}\mapsto \frac{1}{K_{u}^{2}}u_{n}$. By (\ref{8}), $\|T\|\leq 1$ and $\|T^{-1}\|\leq \frac{K_{u}^{3}}{c}.$
Hence $\alpha_{c_{0}}(X)\geq \frac{c}{K^{3}_{u}}$. The arbitrariness of $c$ completes the proof.
\end{proof}

\section{Quantifications of reflexivity}

\begin{theorem}\label{5.1}
Let $X$ be Banach space with a basis $(x_{n})_{n=1}^\infty$. Then
\begin{enumerate}[before=\itshape,font=\normalfont]
    \item $\operatorname{bc}_{3}((x_{n})_{n=1}^\infty)\leq 2K^{2}\operatorname{wk}_{X}(B_{X}).$
    \item $\operatorname{sh}((x_{n})_{n=1}^\infty)\leq 4K^{3}\operatorname{wk}_{X}(B_{X}).$
\end{enumerate}
\end{theorem}
\begin{proof}
(1). Let $0<c<\operatorname{bc}_{3}((x_{n})_{n=1}^\infty)$. Then there exist $x^{**}\in B_{X^{**}}$ and a strictly increasing sequence $(k_{n})_{n=1}^\infty$ so that
\[
    \Big\|\sum_{i=k_{2n-1}+1}^{k_{2n}}\langle x^{**},x^{*}_{i}\rangle x_{i}\Big\|>c \quad (n\in \mathbb N).
\]
By \eqref{2}, for each $n$ we choose $f_{n}\in B_{V}$ so that
\[
    \Big|\big\langle f_{n},\sum_{i=k_{2n-1}+1}^{k_{2n}}\langle x^{**},x^{*}_{i}\rangle x_{i}\big\rangle\Big|>\frac{c}{K} \quad (n\in \mathbb N).
\]
We \emph{claim} that $\textrm{d}(x^{**},X)\geq \frac{c}{2K^{2}}$.

Indeed, for every $x\in X$, we get
\begin{align*}
2K\|x^{**}-x\|&\geq |\langle x^{**},\sum_{i=k_{2n-1}+1}^{k_{2n}}\langle f_{n},x_{i}\rangle x^{*}_{i}\rangle -\langle \sum_{i=k_{2n-1}+1}^{k_{2n}}\langle f_{n},x_{i}\rangle x^{*}_{i},x\rangle|\\
&=|\langle f_{n},\sum_{i=k_{2n-1}+1}^{k_{2n}}\langle x^{**},x^{*}_{i}\rangle x_{i}\rangle-\langle \sum_{i=k_{2n-1}+1}^{k_{2n}}\langle f_{n},x_{i}\rangle x^{*}_{i},x\rangle|\\
&\geq \frac{c}{K}-|\langle f_{n},\sum_{i=k_{2n-1}+1}^{k_{2n}}\langle x^{*}_{i},x\rangle x_{i}\rangle|\\
&\geq \frac{c}{K}-\|\sum_{i=k_{2n-1}+1}^{k_{2n}}\langle x^{*}_{i},x\rangle x_{i}\|.
\end{align*}
Letting $n\rightarrow \infty$, we get $2K\|x^{**}-x\|\geq \frac{c}{K}$. This proves the claim. Consequently, we get $\operatorname{wk}_{X}(B_{X})\geq \frac{c}{2K^{2}}$. The arbitrariness of $c$ completes the proof.

(2). Combining Theorem \ref{4.6}, Theorem \ref{4.4} and (1), we get
\begin{equation}\label{12}
\operatorname{sh}((x_{n})_{n=1}^\infty)\leq 2K^{3}\operatorname{wk}_{V}(B_{V}).
\end{equation}
By (\ref{14}),
\begin{equation}\label{13}
\operatorname{wk}_{V}(B_{V})\leq 2\operatorname{wk}_{X^{*}}(B_{X^{*}})
\end{equation}
By (\ref{12}), (\ref{13}) and (\ref{15}), we get
\[
    \operatorname{sh}((x_{n})_{n=1}^\infty)\leq 4K^{3}\operatorname{wk}_{X^{*}}(B_{X^{*}})=4K^{3}\operatorname{wk}_{X}(B_{X}).
\]
This completes the proof.\end{proof}

\begin{theorem}\label{5.4}
Let $X$ be a Banach space with a basis $(x_{n})_{n=1}^\infty$.
\begin{enumerate}[before=\itshape,font=\normalfont]
    \item If $(x_{n})_{n=1}^\infty$ is boundedly complete, then
    \[
        \operatorname{wck}_{X}(B_{X})\leq (K+1)\widehat{\operatorname{d}}(B_{X^{*}},V).
    \]
    \item If $(x_{n})_{n=1}^\infty$ is shrinking, then
    \[
        \operatorname{wck}_{X}(B_{X})\leq (K+1)^{2}\operatorname{bc}_{2}((x_{n})_{n=1}^\infty)
    \]
\end{enumerate}
\end{theorem}
\begin{proof}
(1). Let $0<c<\operatorname{wck}_{X}(B_{X})$. Then there exists a sequence $(y_{n})_{n}$ in $B_{X}$ so that $\operatorname{d}(\textrm{clust}_{X^{**}}((y_{n})_{n}),X)>c.$ Take $x^{**}_{0}\in \textrm{clust}_{X^{**}}((y_{n})_{n})$. We may choose a strictly increasing sequence $(k_{n})_{n=1}^\infty$ so that $|\langle x^{**}_{0}-y_{k_{n}},x^{*}_{i}\rangle|<\frac{1}{n}$  $(i=1,2,\ldots,n)$ This implies that $\lim_{n\rightarrow \infty}\langle x^{*}_{i},y_{k_{n}}\rangle=\langle x^{**}_{0},x^{*}_{i}\rangle$ for each $i$. Note that, for each $m$, we get
\[
    \|\sum_{i=1}^{m}\langle x^{**}_{0},x^{*}_{i}\rangle x_{i}\|=\lim_{n\rightarrow \infty}\|\sum_{i=1}^{m}\langle x^{*}_{i},y_{k_{n}}\rangle x_{i}\|\leq K.
\]
By the hypothesis, $\sum_{i=1}^{\infty}\langle x^{**}_{0},x^{*}_{i}\rangle x_{i} =x_{0}$ for some $x_{0}\in X$. Moreover, $\|x_{0}\|\leq K$. Hence $\|x^{**}_{0}-x_{0}\|>c$. Take $x_{0}^{*}\in B_{X^{*}}$ so that $|\langle x^{**}_{0}-x_{0}, x^{*}_{0}\rangle|>c$. By the definition of $x_{0}$, we get $\langle x^{*}_{n},x_{0}\rangle=\langle x^{**}_{0},x^{*}_{n}\rangle$ for all $n$ and so $\langle x^{**}_{0}-x_{0}, x^{*}\rangle=0$ for all $x^{*}\in V$. Thus, for all $x^{*}\in V$, we get
\[
    (K+1)\|x^{*}_{0}-x^{*}\|\geq |\langle x^{**}_{0}-x_{0}, x^{*}_{0}-x^{*}\rangle|>c.
\]
This implies that
\[
    (K+1)\widehat{\operatorname{d}}(B_{X^{*}},V)\geq(K+1)\textrm{d}(x^{*}_{0},V)\geq c.
\]
As $c$ was arbitrarty, the proof of (1) is complete.

(2). Suppose that $(x_{n})_{n=1}^\infty$ is shrinking. It follows from \cite[Theorem 3]{James:1950} that $(x^{*}_{n})_{n=1}^{\infty}$ is a boundedly complete basis for $X^{*}$. By (1) and Theorem \ref{2.1}, we get
\[
    \operatorname{wck}_{X^{*}}(B_{X^{*}})\leq (K+1)^{2}\operatorname{sh}((x^{*}_{n})_{n=1}^{\infty}).
\]
By Theorem \ref{4.5},
\[
    \operatorname{wck}_{X^{*}}(B_{X^{*}})\leq (K+1)^{2}\operatorname{bc}_{2}((x_{n})_{n=1}^\infty).
\]
By \eqref{15}, we arrive at the conclusion.\end{proof}

\begin{theorem}\label{5.5}
Let $X$ be a Banach space with an unconditional basis $(x_{n})_{n=1}^\infty$.
\begin{enumerate}[before=\itshape,font=\normalfont]
    \item If $X$ contains no isomorphic copies of $\ell_{1}$, then
    \[
        \frac{1}{K_{u}^{3}K(K+1)^{2}}\operatorname{wck}_{X}(B_{X})\leq \alpha_{c_{0}}(X)\leq \alpha_{\ell_{1}}(X^{*})\leq\operatorname{wck}_{X}(B_{X}).
    \]
    \item If $X$ contains no isomorphic copies of $c_{0}$, then
    \[
        \frac{1}{K_{u}(K+1)^{2}}\operatorname{wck}_{X}(B_{X})\leq \alpha_{\ell_{1}}(X)\leq \operatorname{wck}_{X}(B_{X}).
    \]
\end{enumerate}
\end{theorem}
\begin{proof}
(1). By Theorem \ref{1.1}, $(x_{n})_{n=1}^\infty$ is shrinking. Combining Theorem \ref{4.7}, Theorem \ref{4.4} and Theorem \ref{5.4} (2), we get
\[
    \alpha_{c_{0}}(X)\geq     \frac{1}{K_{u}^{3}}\operatorname{bc}_{1}((x_{n})_{n=1}^\infty)\geq \frac{1}{K_{u}^{3}K}\operatorname{bc}_{2}((x_{n})_{n=1}^\infty)
    \geq \frac{1}{K_{u}^{3}K(K+1)^{2}}\operatorname{wck}_{X}(B_{X}).
\]
The second and third inequalities of (1) follow from Lemma \ref{2.1.2} and (\ref{15}).

(2). The right inequality of (2) follows from Lemma \ref{2.1.2}. By Theorem \ref{1.2}, $(x_{n})_{n=1}^\infty$ is boundedly complete. By Theorem \ref{5.4}, Theorem \ref{2.1} and Theorem \ref{2.3}, we get
\[
    \operatorname{wck}_{X}(B_{X})\leq (K+1)\widehat{\operatorname{d}}(B_{X^{*}},V)
    \leq (K+1)^{2}\operatorname{sh}((x_{n})_{n=1}^\infty)\leq (K+1)^{2}K_{u}\alpha_{\ell_{1}}(X).
\]
The proof is complete.\end{proof}

\begin{theorem}
Let $X$ be a Banach space with an unconditional basis. Then
\[
    \frac{1}{K_{u}^{3}K(K+1)^{2}}\operatorname{wck}_{X}(B_{X})\leq\operatorname{sep}(B_{X^{**}})\leq \operatorname{wk}_{X}(B_{X}).
\]
\end{theorem}
\begin{proof}
Let $(x_{n})_{n=1}^\infty$ be an unconditional basis for $X$.
Let $c>\operatorname{wk}_{X}(B_{X})$ be arbitrary. Let $\mathcal Q$ be a countable dense subset of $\mathbb R$ and $C=\{\sum_{i=1}^{n}r_{i}x_{i}\colon n\in \mathbb{N},r_{1},r_{2},\ldots,r_{n}\in \mathcal{Q}\}$.
It is easy to see that $B_{X^{**}}\subseteq C+cB_{X^{**}}.$ Hence $\operatorname{sep}(B_{X^{**}})\leq c$. As $c$ was arbitrary, the proof of the second inequality is complete.

For the first inequality, we divide the proof into two cases. If $X$ contains an isomorphic copy of $\ell_{1}$, it follows from James' distortion theorem that $\alpha_{\ell_{1}}(X)=1$. By Theorem \ref{2.2}, we get $\operatorname{sep}(B_{X^*})=1$. By Lemma \ref{2.1.1}, $\operatorname{sep}(B_{X^{**}})=1$. The first inequality clearly holds. If $X$ contains no isomorphic copy of $\ell_{1}$, we get, by Theorem \ref{5.5} and Theorem \ref{2.2},
\[
    \frac{1}{K_{u}^{3}K(K+1)^{2}}\operatorname{wck}_{X}(B_{X})\leq \alpha_{\ell_{1}}(X^{*})\leq \operatorname{sep}(B_{X^{**}}).
\]
This completes the proof.\end{proof}

\end{document}